\documentclass[a4paper]{article}			
\usepackage[latin1]{inputenc}				
\usepackage[T1]{fontenc}					
\usepackage{lmodern}						
\usepackage[english]{babel}					
\usepackage{amsmath,amsfonts}				
\usepackage{amssymb}						
\usepackage{stmaryrd}                       
\usepackage{multirow}						

\usepackage{bm}								
\usepackage{amsthm}

\usepackage{a4wide} 						

\newtheorem{thm}{Theorem}
\newtheorem{lem}{Lemma}
\newtheorem{prop}{Proposition}

\newtheorem{coro}{Corollary}
\newtheorem{defi}{Definition}

\theoremstyle{definition}					
\newtheorem{rem}{Remark}

\newcommand{\R}{\mathbb{R}}

\newcommand{\N}{\mathbb{N}}
\newcommand{\Z}{\mathbb{Z}}
\newcommand{\Q}{\mathbb{Q}}

\title{Liminf Sets in Simultaneous Diophantine Approximation}
\author{Faustin ADICEAM\\
   Department of Mathematics, Logic House,\\ 
   National University of Ireland at Maynooth\\
   email~: \texttt{fadiceam@gmail.com}}
\date{}

\begin{document}
\maketitle

\begin{abstract}
Let $\mathcal{Q}$ be an infinite set of positive integers. Denote by $W^*_{\tau, n}(\mathcal{Q})$ the set of $n$--tuples of real numbers simultaneously $\tau$--well approximable by infinitely many rationals with denominators in $\mathcal{Q}$ but by only finitely many rationals with denominators in the complement of $\mathcal{Q}$. The Hausdorff dimension of the liminf set $W^*_{\tau, n}(\mathcal{Q})$ is computed when $n\ge 1$ and $\tau >2+1/n$. A $p$--adic analogue of the problem is also studied.
\end{abstract}

\section{Introduction and statement of the result}

Let $n \ge 1$ be an integer and $\tau > 1$ a real number. Given an infinite set of positive integers $\mathcal{Q}$, denote by $W_{\tau, n}(\mathcal{Q})$ the set of points in dimension $n\ge 1$ approximable at order $\tau$ by infinitely many rationals with denominators in $\mathcal{Q}$, i.e. the limsup set 
\begin{equation}\label{enslimsup}
W_{\tau, n}(\mathcal{Q}) := \left\{ \bm{x}\in \R^n \; : \; \left|\bm{x}-\bm{p}/q\right|<q^{-\tau} \; \; \mbox{ for  i.m. } (\bm{p},q) \in \Z^n\times\mathcal{Q} \right\}.
\end{equation} 
Here and throughout, \textit{i.m.} stands for \textit{infinitely many}, $\left|\bm{x}\right|$ is the usual supremum norm of a vector $\bm{x}\in\R^n$ and $\bm{p}/q$ is shorthand notation for the rational vector $(p_1/q, \dots , p_n/q)$, where $\bm{p} = (p_1, \dots , p_n)\in\Z^n$.

\sloppy Jarn{\'\i}k in~\cite{multijar} and Besicovitch in~\cite{besico} proved independently that the Hausdorff dimension $\dim W_{\tau, n}(\N)$ of the set $W_{\tau, n}(\N)$ was equal to $(n+1)/\tau$ as soon as $\tau>1+1/n$. Subsequently, Borosh and Fraenkel generalized this result in~\cite{borofraen} to the case of any infinite subset $\mathcal{Q}\subset \N$ by showing that $\dim W_{\tau, n}(\mathcal{Q})=(n+\nu(\mathcal{Q}))/\tau$ when $\tau>1+\nu(\mathcal{Q})/n$, where $\nu(\mathcal{Q})$ is the exponent of convergence of $\mathcal{Q}$ defined as 
\begin{equation}\label{expconvgce}
\nu(\mathcal{Q}) := \inf \left\{\nu >0 \; : \; \sum_{q\in \mathcal{Q}} q^{-\nu} < \infty \right\} \in [0,1].
\end{equation}

On the other hand, the corresponding liminf set 
\begin{equation}\label{limfinreel}
W^*_{\tau, n}(\mathcal{Q}) := W_{\tau, n}(\N)\, \backslash \, W_{\tau, n}\left(\N\backslash \mathcal{Q}\right) = W_{\tau, n}(\mathcal{Q})\, \backslash \, W_{\tau, n}\left(\N\backslash \mathcal{Q}\right)
\end{equation}
has received much less attention. Explicitly, this is the set of all those vectors $\bm{x}$ in $\R^n$ which admit infinitely many approximations at order $\tau$ as in~(\ref{enslimsup}) by rational vectors $(\bm{p},q)$ whose denominators $q$ lie in $\mathcal{Q}$, but only finitely many approximations by rational vectors whose denominators do not lie in $\mathcal{Q}$. The author considered in~\cite{noteliminfmoi} the case where the set $Q$ was a so--called $\N\backslash \mathcal{Q}$--free set (that is, a set $\mathcal{Q}$ whose elements are divisible by no integer in the complement of $\mathcal{Q}$) and exhibited a non--trivial lower bound for $\dim W^*_{\tau, n}(\mathcal{Q})$ when $n\ge 2$ and $\tau > 1+1/(n-1)$. He also provided a construction, explicit in terms of the continued fraction expansion, of uncountably many Liouville numbers lying in the set $W^*_{\tau, 1}(d\N)$, where $d\ge 2$ is any integer and $\tau>2$.

It is not clear that the set $W^*_{\tau, n}(\mathcal{Q})$ should be non--empty for a general infinite subset $\mathcal{Q}\subset\N$. This is in particular implied by the following much stronger statement which is the main result of this paper~:

\begin{thm}\label{thmprinciliminf}
Let $\mathcal{Q}\subset\N$ be infinite. Assume that $n\ge 1$ is an integer and that $\tau > 2+1/n$ is a real number. Then $$\dim W^*_{\tau, n}(\mathcal{Q}) = \frac{n+\nu(\mathcal{Q})}{\tau}\cdot$$
\end{thm} 
Thus when $\tau > 2+1/n$, the limsup set $W_{\tau, n}(\mathcal{Q})$ and the associated liminf set $W^*_{\tau, n}(\mathcal{Q})$ actually share the same Hausdorff dimension. This leaves a gap corresponding to the case where $\tau$ lies in the interval $(1+\nu(\mathcal{Q})/n \, ,\, 2+1/n]$. The nature of this restriction shall clearly appear in the course of the proof and shall then be discussed. It is however worth mentioning at this stage that the underlying difficulty does not seem easy to overcome and may be linked to some deep problems in the metric theory of numbers.

\paragraph{\textbf{\large{Notation}}\\}
In addition to those already introduced, the following pieces of notation shall be used throughout~:

\begin{itemize}
\item $x \ll y$ (resp. $x\gg y$, where $x,y\in \R$)~: there exists a constant $c>0$ such that $x\le cy$ (resp. $x\ge cy$).
\item $x\asymp y$ ($x,y\in \R$) means both $x \ll y$ and $x\gg y$.
\item $ \llbracket x , y \rrbracket$ ($x, y\in\R$, $x \le y$)~: interval of integers, i.e. $\llbracket x , y \rrbracket = \left\{ n\in\Z \; : \; x\le n \le y\right\}$.
\item $\lambda_n$~: the $n$--dimensional Lebesgue measure (for simplicity, $\lambda := \lambda_1$). 
\item $\# X$~: the cardinality of a finite set $X$.
\item $|U|$~: the diameter of a bounded set $U\subset \R^n$.
\item $\delta_n\left(\mathcal{S}\right):=\#\left(\mathcal{S}\cap \llbracket 1 , n \rrbracket  \right)$ for any subset $\mathcal{S}\subset\N$.
\item $I_{\tau}\!\!\left(\frac{p}{q}\right):=\left(\frac{p}{q}-\frac{1}{q^{\tau}}\, ,\, \frac{p}{q}+\frac{1}{q^{\tau}}\right)$, where $\frac{p}{q}\in\Q$.
\item $C_{\tau}\!\!\left(\frac{\bm{p}}{q}\right):=\prod_{i=1}^{n}I_{\tau}\!\!\left(\frac{p_i}{q}\right)$, where $\bm{p}=(p_1,\dots,p_n)\in\Z^n$ and $q\in\N$ (note that with this convention, $C_{\tau}\!\!\left(\frac{2\bm{p}}{2q}\right)$ is strictly contained in $C_{\tau}\!\!\left(\frac{\bm{p}}{q}\right)$).
\end{itemize}

\section{Auxiliary lemmas}

In this section, $\mathcal{S}$ denotes an arbitrary infinite set of natural numbers.

\subsection{On the logarithmic density of a subset of integers}

As is well--known, the exponent of convergence, as defined by~(\ref{expconvgce}), of the set $S$ is related to its logarithmic density in the following way (see for instance~\cite{hardyriesz} for a proof)~: 
\begin{equation}\label{expconvlogdens}
\nu(\mathcal{S})\, = \, \limsup_{n\rightarrow +\infty}\left(\frac{\log \delta_n\left(\mathcal{S}\right)}{\log n}\right).
\end{equation}
The next lemma provides a similar formula for $\nu(\mathcal{S})$.

\begin{lem}\label{autreversiondensitelogexpconv}
The following equation holds~: $$\nu(\mathcal{S})\, = \, \limsup_{n\rightarrow +\infty}\left(\frac{\log \left( \delta_{2n}\left(\mathcal{S}\right)-\delta_{n}\left(\mathcal{S}\right)\right)}{\log n}\right).$$
\end{lem}

\begin{proof}
First note that, for $n\in\N$, $$\frac{\log \left( \delta_{2n}\left(\mathcal{S}\right)-\delta_{n}\left(\mathcal{S}\right)\right)}{\log n} \; \le \; \frac{\log \delta_{2n}\left(\mathcal{S}\right)}{\log n} \; \underset{n\rightarrow +\infty}{\sim}  \; \frac{\log \delta_{2n}\left(\mathcal{S}\right)}{\log 2n}\cdotp$$
Taking the limsup on both sides of this inequality, it is easily seen that~(\ref{expconvlogdens}) implies $$\limsup_{n\rightarrow +\infty} \left(\frac{\log \left( \delta_{2n}\left(\mathcal{S}\right)-\delta_{n}\left(\mathcal{S}\right)\right)}{\log n} \right) \; \le \; \nu(\mathcal{S}).$$
This suffices to prove the result in the case $\nu(\mathcal{S})=0$ since, the set $\mathcal{S}$ being infinite, $\delta_{2n}\left(\mathcal{S}\right)-\delta_{n}\left(\mathcal{S}\right) \ge 1$ for infinitely many $n\in\N$. Therefore, assume from now on that $\nu(\mathcal{S})>0$. Then~(\ref{expconvlogdens}) shows the existence of a sequence $\left(n_k\right)_{k\ge 0}$ of positive integers such that 
\begin{equation}\label{asymlogdensity}
\log \delta_{n_k}\left(\mathcal{S}\right) \; \underset{n\rightarrow +\infty}{\sim} \; \nu\left( \mathcal{S}\right)\, \log n_k.
\end{equation}
For a fixed $k\in\N$, consider the following partition of the interval $\llbracket 2 , n_k \rrbracket$ into $u_k:=\lfloor \log n_k/\log 2\rfloor$ subintervals~: $$\left\llbracket 2 , n_k \right\rrbracket = \bigcup_{r=0}^{u_k}\bigg\llbracket \frac{n_k}{2^{r+1}}+1 , \frac{n_k}{2^r} \bigg\rrbracket.$$
From the definition of the integer $\delta_{n_k}\left(\mathcal{S}\right)$, at least one of these intervals contains more than $(\delta_{n_k}\left(\mathcal{S}\right)-1)/(u_k+1)$ elements of $\mathcal{S}$, which determines a rational number $l_k$ of the form $n_k/2^{a+1}$ ($0\le a \le u_k -1$) such that 
\begin{align}\label{logsub1}
\frac{\delta_{n_k}\left(\mathcal{S}\right)-1}{u_k+1} \; &\le \; 2l_k-(l_k+1)+1=l_k \; \le \; \frac{n_k}{2} \quad \textrm{ and } \quad \frac{\delta_{n_k}\left(\mathcal{S}\right)-1}{u_k+1} \; &\le \; \delta_{2l_k}\left(\mathcal{S} \right)-\delta_{l_k}\left(\mathcal{S} \right).
\end{align}
From~(\ref{asymlogdensity}) and from the definition of $u_k$, one deduces on the one hand that the first inequality in~(\ref{logsub1}) implies that the sequence $\left(l_k \right)_{k\ge1}$ tends to infinity and that, on the other, $$\frac{\log\left(\delta_{n_k}\left(\mathcal{S}\right)-1\right)-\log\left(u_k+1\right)}{\log n_k} \; \underset{n\rightarrow +\infty}{\sim}  \; \frac{\log \delta_{n_k}\left(\mathcal{S}\right)}{\log n_k} \; \underset{n\rightarrow +\infty}{\sim}  \nu(\mathcal{S}).$$ Furthermore, it follows from~(\ref{logsub1}) that $$\frac{\log\left(\delta_{n_k}\left(\mathcal{S}\right)-1\right)-\log\left(u_k+1\right)}{\log n_k} \; \le \; \frac{\log\left(\delta_{2l_k}\left(\mathcal{S}\right)-\delta_{l_k}\left(\mathcal{S}\right)\right)}{\log l_k}\cdotp$$ Combining these last two inequalities leads to the relationship $$\nu\left( \mathcal{S}\right)\; \le \; \limsup_{n\rightarrow +\infty} \left(\frac{\log\left(\delta_{2l_k}\left(\mathcal{S}\right)-\delta_{l_k}\left(\mathcal{S}\right)\right)}{\log l_k}\right),$$ which completes the proof.
\end{proof}

One key--step in the proof of Theorem~\ref{thmprinciliminf} is to approach an infinite set of positive integers by arbitrarily large subsets, the size of a subset being measured by its exponent of convergence. In this respect, the following proposition will turn out to be very useful.

\begin{prop}\label{sousensbiendistribue}
Assume that $\nu\left(\mathcal{S}\right)>0$ and let $\nu\in \left(0 , \nu\left(\mathcal{S}\right)\right)$. Furthermore, let $\left(\alpha_n\right)_{n\ge 0}$ be a sequence of positive reals such that the sequence $\left(n^{\nu}\alpha_n\right)_{n\ge 0}$ is increasing and such that $\left(\log \alpha_n / \log n\right)_{n\ge 2}$ tends to 0 as $n$ goes to infinity.

Then, there exists a subset $\mathcal{S}_{\nu}\subset \mathcal{S}$ such that~:
\begin{itemize}
\item for all $n\ge 1$, $\delta_{2n}\left(\mathcal{S}_{\nu}\right)-\delta_n\left(\mathcal{S}_{\nu}\right) \le n^{\nu}\alpha_n.$
\item there exists a strictly increasing sequence of positive integers $\left(n_k\right)_{k\ge 0}$ satisfying $$\delta_{2n_k}\left(\mathcal{S}_{\nu}\right)-\delta_{n_k}\left(\mathcal{S}_{\nu}\right) \; \underset{k\rightarrow +\infty}{\sim} \; n_k^{\nu}\,\alpha_{n_k}.$$
\end{itemize}
In particular, $\nu\left(\mathcal{S}_{\nu}\right) = \nu$.
\end{prop}

\begin{proof}
The fact that $\nu\left(\mathcal{S}_{\nu}\right) = \nu$ follows immediately from Lemma~\ref{autreversiondensitelogexpconv}. Note that this lemma applied to the set $\mathcal{S}$ amounts to claiming the existence of a sequence of real numbers $\left(\beta_n\right)_{n\ge 0}$ tending to zero and of a strictly increasing sequence of positive integers $\left(p_k\right)_{k\ge 0}$ satisfying
\begin{align}\label{memequavantpourcettedemo}
\delta_{2n}\left(\mathcal{S}\right)-\delta_n\left(\mathcal{S}\right) \le n^{\nu\left(\mathcal{S}\right)+\beta_n} \,\textrm{ for all } n\in\N \,\textrm{ and }\, \delta_{2p_k}\left(\mathcal{S}\right)-\delta_{p_k}\left(\mathcal{S}\right) \; \underset{k\rightarrow +\infty}{\sim} \; p_k^{\nu\left(\mathcal{S}\right)+\beta_{p_k}}.  
\end{align}
Note also that the assumption that $\log \alpha_n / \log n$ tends to zero amounts to the fact that $\alpha_n = o\left(n^{\epsilon}\right)$ for all $\epsilon >0$. Thus, the second relationship in~(\ref{memequavantpourcettedemo}) and the fact that $\nu<\nu\left(\mathcal{S}\right)$ guarantee the existence of a smallest positive integer $n_1$ such that $$\lfloor n_1^{\nu}\,\alpha_{n_1}\rfloor \; < \; \delta_{2n_1}\left(\mathcal{S}\right)-\delta_{n_1}\left(\mathcal{S}\right):= r_1.$$
Now remove $r_1-\lfloor n_1^{\nu}\,\alpha_{n_1}\rfloor$ elements of $\mathcal{S}$ from the interval $\llbracket n_1+1 , 2n_1 \rrbracket$ to define a subset $\mathcal{S}_{\nu}^{(1)}\subset \mathcal{S}$ satisfying the following properties~:
\begin{itemize}
\item $\mathcal{S}_{\nu}^{(1)}$ and $\mathcal{S}$ coincide on the intervals $\llbracket 1 , n_1 \rrbracket$ and $\N\backslash \llbracket 1 , 2n_1 \rrbracket$,
\item for all $n \in \llbracket 1 , n_1 \rrbracket$, $\delta_{2n}\left(\mathcal{S}_{\nu}^{(1)}\right)-\delta_{n}\left(\mathcal{S}_{\nu}^{(1)}\right) \; \le \; n^{\nu} \alpha_{n},$
\item $\delta_{2n_1}\left(\mathcal{S}_{\nu}^{(1)}\right)-\delta_{n_1}\left(\mathcal{S}_{\nu}^{(1)}\right) \, = \, \lfloor n_1^{\nu}\, \alpha_{n_1} \rfloor$.
\end{itemize}

Consider then the smallest integer $n_2 > n_1$ such that $$\lfloor n_2^{\nu}\,\alpha_{n_2}\rfloor \; < \; \delta_{2n_2}\left(\mathcal{S}_{\nu}^{(1)}\right)-\delta_{n_2}\left(\mathcal{S}_{\nu}^{(1)}\right):= r_2.$$ Since for $n\ge 2n_1+1$, $\delta_{2n}\left(\mathcal{S}_{\nu}^{(1)}\right)-\delta_{n}\left(\mathcal{S}_{\nu}^{(1)}\right) = \delta_{2n}\left(\mathcal{S}\right)-\delta_{n}\left(\mathcal{S}\right),$ the existence of $n_2$ is guaranteed in the same way as for $n_1$. 

Defining $u_2:= \max\left\{n_2, 2n_1 \right\}$, remove $r_2-\lfloor n_2^{\nu}\,\alpha_{n_2}\rfloor$ elements of $\mathcal{S}$ from the interval $\llbracket u_1+1 , 2n_2 \rrbracket$. This is clearly possible if $n_2\ge 2n_1$ as there is no overlap in this case between the intervals $\llbracket n_1 , 2n_1 \rrbracket$ and $\llbracket u_1+1 , 2n_2 \rrbracket$. But this is also possible if $n_1 < n_2 < 2n_1$~: indeed, if the interval $\llbracket u_1+1 , 2n_2 \rrbracket = \llbracket 2n_1+1 , 2n_2 \rrbracket$ contained strictly less than $r_2-\lfloor n_2^{\nu}\,\alpha_{n_2}\rfloor$ elements, one would have~:
\begin{align*}
r_2 := \delta_{2n_2}\left(\mathcal{S}_{\nu}^{(1)}\right)-\delta_{n_2}\left(\mathcal{S}_{\nu}^{(1)}\right) &= \delta_{2n_2}\left(\mathcal{S}_{\nu}^{(1)}\right)-\delta_{2n_1}\left(\mathcal{S}_{\nu}^{(1)}\right)+\delta_{2n_1}\left(\mathcal{S}_{\nu}^{(1)}\right)-\delta_{n_2}\left(\mathcal{S}_{\nu}^{(1)}\right)\\
&= \delta_{2n_2}\left(\mathcal{S}\right)-\delta_{2n_1}\left(\mathcal{S}\right)+\delta_{2n_1}\left(\mathcal{S}_{\nu}^{(1)}\right)-\delta_{n_2}\left(\mathcal{S}_{\nu}^{(1)}\right) \\
& \qquad\qquad \left(\textrm{as  } \mathcal{S}_{\nu}^{(1)}\cap \left\{ n\ge 2n_1+1\right\}= \mathcal{S}\cap \left\{ n\ge 2n_1+1\right\}\right)\\
&\le \delta_{2n_2}\left(\mathcal{S}\right)-\delta_{2n_1}\left(\mathcal{S}\right)+\delta_{2n_1}\left(\mathcal{S}_{\nu}^{(1)}\right)-\delta_{n_1}\left(\mathcal{S}_{\nu}^{(1)}\right) \\
& < r_2-\lfloor n_2^{\nu}\,\alpha_{n_2}\rfloor + \lfloor n_1^{\nu}\,\alpha_{n_1}\rfloor\\
& \le r_2
\end{align*}
since the sequence $\left(n^{\nu}\alpha_n\right)_{n\ge 0}$ is increasing. This contradiction shows that one can find a subset $\mathcal{S}_{\nu}^{(2)}\subset \mathcal{S}_{\nu}^{(1)}$ such that~:
\begin{itemize}
\item $\mathcal{S}_{\nu}^{(1)}$ and $\mathcal{S}_{\nu}^{(2)}$ coincide on the intervals $\llbracket 1 , n_2 \rrbracket$ and $\N\backslash \llbracket 1 , 2n_2 \rrbracket$,
\item for all $n \in \llbracket 1 , n_2 \rrbracket$, $\delta_{2n}\left(\mathcal{S}_{\nu}^{(2)}\right)-\delta_{n}\left(\mathcal{S}_{\nu}^{(2)}\right) \; \le \; n^{\nu} \alpha_{n},$
\item $\delta_{2n_2}\left(\mathcal{S}_{\nu}^{(2)}\right)-\delta_{n_2}\left(\mathcal{S}_{\nu}^{(2)}\right) \, = \, \lfloor n_2^{\nu}\, \alpha_{n_2} \rfloor$.
\end{itemize}

By induction, one can thus construct a decreasing sequence $\left(\mathcal{S}_{\nu}^{(k)}\right)_{k\ge 1}$ of subsets of $\mathcal{S}$ and a strictly increasing sequence of natural integers $\left(n_k\right)_{k\ge 1}$ such that, for all $k\ge 2$, 
\begin{itemize}
\item $\mathcal{S}_{\nu}^{(k-1)}$ and $\mathcal{S}_{\nu}^{(k)}$ coincide on $\llbracket 1 , n_k \rrbracket$ and $\N\backslash \llbracket 1 , 2n_k \rrbracket$,
\item for all $n \in \llbracket 1 , n_k \rrbracket$, $\delta_{2n}\left(\mathcal{S}_{\nu}^{(k)}\right)-\delta_{n}\left(\mathcal{S}_{\nu}^{(k)}\right) \; \le \; n^{\nu} \alpha_{n},$
\item $\delta_{2n_k}\left(\mathcal{S}_{\nu}^{(k)}\right)-\delta_{n_k}\left(\mathcal{S}_{\nu}^{(k)}\right) \, = \, \lfloor n_k^{\nu}\, \alpha_{n_k} \rfloor$.
\end{itemize}
By construction, the set $\mathcal{S}_{\nu} := \cap_{k=1}^{+\infty} \mathcal{S}_{\nu}^{(k)}$ satisfies the conclusion of the proposition.
\end{proof}

\subsection{Steps to the construction of a Cantor set}

Theorem~\ref{thmprinciliminf} will be proved by exhibiting nice Cantor sets contained in the liminf set under consideration. To this end, a few auxiliary results are gathered in this subsection. They are preceded by two definitions which shall be used throughout this paper.

\begin{defi}\label{vecteurprimitif}
A vector $\bm{p}=(p_1, \dots, p_n)\in\Z^n$ is $q$\emph{--primitive} (where $q\in\N$) if at least one of the components $p_i$ of $\bm{p}$ is coprime to $q$. The vector $\bm{p}$ is \emph{absolutely} $q$\emph{--primitive} if all its components are coprime to $q$.
\end{defi}

\begin{defi}
Given $\tau >1$, $\bm{p_0}\in\Z^n$ and $q_0\in\N$, a \emph{hypercube of new generation in $C_{\tau}\!\!\left(\frac{\bm{p_0}}{q_0}\right)$} is a hypercube of the form $C_{\tau}\!\!\left(\frac{\bm{p}}{q}\right)$ contained in $C_{\tau}\!\!\left(\frac{\bm{p_0}}{q_0}\right)$ such that $\bm{p}\in\Z^n$ is absolutely $q$--primitive ($q\in\N$) and such that for any $q_1\in \llbracket q_0+1 , q-1 \rrbracket$ and any $\bm{p_1}\in\Z^n$, $$C_{\tau}\!\!\left(\frac{\bm{p_1}}{q_1}\right) \cap C_{\tau}\!\!\left(\frac{\bm{p}}{q}\right) = \emptyset .$$
\end{defi}

Thus, the concept of a hypercube of new generation renders the idea that such a polytope covers a volume inside a given hypercube which has been covered by no other. The next proposition counts the number of such hypercubes and constitutes a problem specific to the liminf setup in Diophantine approximation. It is preceded by a well--known lemma on the repartition of integers coprime to a given natural number.

\begin{lem}\label{distrinbscopremiers}
Let $q$ be a positive integer and $\eta$ be any positive real number. Denote by $\varphi_{\eta}(q)$ the number of integers less than $\eta q$ and coprime to $q$. Then, for any $\epsilon >0$, $$\varphi_{\eta}(q) = \varphi(q) \left(\eta + o\left(q^{-1+\epsilon}\right)\right),$$ where $\varphi$ denotes Euler's totient function.

In particular, if $\epsilon\in (0,1)$, $\eta > q^{-1+\epsilon}$ and $q$ is large enough, then for any $\gamma \ge 0$,$$\# \left\{\, p\in\llbracket \gamma q , (\gamma +\eta)q \rrbracket \; : \; \gcd(p,q)=1 \, \right\} \;\asymp \; \eta \varphi(q),$$ where the implicit constants depend only on $\epsilon$.
\end{lem} 

\begin{proof}
This follows easily from the inclusion--exclusion principle and some standard estimates of arithmetical functions. See for instance Lemma~III of~\cite{dufsch} for details.
\end{proof}

\begin{prop}\label{originedelacontrainte}
Let $\tau>2+1/n$, $\bm{p_0}\in\Z^n$ and $q_0\in\N$. Assume that $C_{\tau}\!\!\left(\frac{\bm{p_0}}{q_0}\right)\subset (0,1)^n$ and that $q>q_0^{\tau^3}$ has been chosen large enough. Denote furthermore by $\mathcal{N}\left(q,\frac{\bm{p_0}}{q_0}, \tau \right)$ the cardinality of the set of hypercubes of new generation in $C_{\tau}\!\!\left(\frac{\bm{p_0}}{q_0}\right)$ of the form $C_{\tau}\!\!\left(\frac{\bm{p}}{q}\right)$ for some $\bm{p}\in\Z^n$.

Then, provided that $q_0$ is larger than some constant (independent of $q$), $$\mathcal{N}\left(q,\frac{\bm{p_0}}{q_0}, \tau \right)\; \ge \; \frac{\varphi(q)^n\, \lambda_n\!\left(C_{\tau}\!\!\left(\frac{\bm{p_0}}{q_0}\right)\right)}{2^{n+1}}\cdotp$$
\end{prop}

\begin{proof}
Set $\widetilde{C}_{\tau}\!\!\left(\frac{\bm{p_0}}{q_0}\right):= C_{\tau}\!\!\left(\frac{\bm{p_0}}{q_0}\right)\Big\backslash C_{\tau}\!\!\left(\frac{2\bm{p_0}}{2q_0}\right)$. If $q>q_0$ is large enough, the number of absolutely $q$--primitive vectors $\bm{p}\in\Z^n$ such that $C_{\tau}\!\!\left(\frac{\bm{p}}{q}\right) \subset \widetilde{C}_{\tau}\!\!\left(\frac{\bm{p_0}}{q_0}\right)$ is certainly bigger than 
\begin{equation}\label{comptagesansrestrictionabspprimitive}
2^n \left(1-\frac{1}{2^{\tau}}\right)^n\frac{\varphi(q)^n}{2^n} \, \lambda_n\!\left(C_{\tau}\!\!\left(\frac{\bm{p_0}}{q_0}\right)\right) \, \underset{(\tau>1)}{\ge}\, \frac{\varphi(q)^n}{2^n} \lambda_n\!\left(C_{\tau}\!\!\left(\frac{\bm{p_0}}{q_0}\right)\right)
\end{equation}
(this follows for instance from Lemma~\ref{distrinbscopremiers}).

Assume now that there exist an integer $q_1>q_0$ and $\bm{p_1}\in\Z^n$ such that $\widetilde{C}_{\tau}\!\!\left(\frac{\bm{p_0}}{q_0}\right) \cap C_{\tau}\!\!\left(\frac{\bm{p_1}}{q_1}\right) \neq \emptyset.$ In particular, $\bm{p_1}/q_1 \neq \bm{p_0}/q_0$, whence $$\frac{1}{q_0q_1} \le \left|\frac{\bm{p_0}}{q_0} - \frac{\bm{p_1}}{q_1} \right| < \frac{2}{q_0^{\tau}}\cdotp$$
This means that, when computing the number of hypercubes $C_{\tau}\!\!\left(\frac{\bm{p}}{q}\right)$ of new generation in $\widetilde{C}_{\tau}\!\!\left(\frac{\bm{p_0}}{q_0}\right)$ ($\bm{p}\in\Z^n$), it suffices to consider those hypercubes of this form which have no overlap with any hypercube of the form $C_{\tau}\!\!\left(\frac{\bm{p_1}}{q_1}\right)$, where $\bm{p_1}\in\Z^n$ and $q_1> q_0^{\tau - 1}/2$. 

Given this, let us now count the number of integer vectors $\bm{p}\in\Z^n$ such that $C_{\tau}\!\!\left(\frac{\bm{p}}{q}\right)$  has a non--empty intersection with a hypercube $C_{\tau}\!\!\left(\frac{\bm{p_1}}{q_1}\right)$ contained in $C_{\tau}\!\!\left(\frac{\bm{p_0}}{q_0}\right)$, where $\bm{p_1}\in\Z^n$ and $\frac{q_0^{\tau - 1}}{2} < q_1 < q$.

\paragraph{}\emph{First case~:} $\frac{q_0^{\tau - 1}}{2} < q_1 \le \frac{q_0^{\tau}}{4}$. Fix an integer $q_1$ in this range. Then there exists at most one integer vector $\bm{p_1}\in\Z^n$ such that $C_{\tau}\!\!\left(\frac{\bm{p_1}}{q_1}\right) \cap C_{\tau}\!\!\left(\frac{\bm{p_0}}{q_0}\right) \neq \emptyset$. Indeed, should there exist another one $\bm{p'_1}\in\Z^n$, one would have $$\frac{1}{q_1}\, \le \, \left|\frac{\bm{p_1}}{q_1} - \frac{\bm{p'_1}}{q_1}\right| \, \le \, \left|\frac{\bm{p_1}}{q_1} - \frac{\bm{p_0}}{q_0}\right| + \left|\frac{\bm{p_0}}{q_0} - \frac{\bm{p'_1}}{q_1}\right| \, < \, \frac{4}{q_0^\tau},$$ contradicting the assumption on $q_1$.

Suppose now that there does exist $\bm{p_1}= \left(p_{1,i}\right)_{1\le i \le n}\in\Z^n$ satisfying $C_{\tau}\!\!\left(\frac{\bm{p_1}}{q_1}\right) \cap C_{\tau}\!\!\left(\frac{\bm{p_0}}{q_0}\right) \neq \emptyset$. If, furthermore, $\bm{p}=\left(p_1, \dots, p_n\right)\in\Z^n$ is an absolutely $q$--primitive vector such that $C_{\tau}\!\!\left(\frac{\bm{p}}{q}\right) \cap C_{\tau}\!\!\left(\frac{\bm{p_1}}{q_1}\right) \neq \emptyset$, then, for any $i\in \llbracket 1 , n \rrbracket$, 
\begin{equation}\label{localisationpietp1i}
\left| p_i - \frac{p_{1,i}}{q_1}q\right| < \frac{2q}{q_1^\tau}\cdotp
\end{equation}
Under the assumption that $q\ge q_0^{\tau^3}$ and $q_1\le q_0^\tau/4$, it follows from Lemma~\ref{distrinbscopremiers} that, if $q_0$ is chosen large enough, the number of such absolutely $q$--primitive vectors $\bm{p}\in\Z^n$ is less than $$K \frac{\varphi(q)^n}{q_1^{n\tau}}$$ for some constant $K>0$ depending on $n$.

Summing over all the possible values of $q_1$, the number of hypercubes $C_{\tau}\!\!\left(\frac{\bm{p}}{q}\right)$ with $\bm{p}\in\Z^n$ absolutely $q$--primitive having a non--empty intersection with a hypercube of the form $C_{\tau}\!\!\left(\frac{\bm{p_1}}{q_1}\right)$ is seen to be less that 
\begin{equation}\label{comptagepremiercas}
K \varphi(q)^n \sum_{q_0^{\tau -1}/2< q_1 \le q_0^{\tau}/4} q_1^{-n\tau} \, \le \, K\varphi(q)^n \sum_{q_1>q_0^{\tau -1}/2} q_1^{-n\tau} \, \le \, c_1 \frac{\varphi(q)^n}{q_0^{(\tau -1)(n\tau -1)}}
\end{equation}
for some $c_1>0$ depending on $\tau$ and $n$.

\paragraph{}\emph{Second case~:} $\frac{q_0^{\tau}}{4} < q_1 < q$. Fix an integer $q_1$ in this range and assume that $C_{\tau}\!\!\left(\frac{\bm{p_1}}{q_1}\right) \cap C_{\tau}\!\!\left(\frac{\bm{p}}{q}\right) \neq \emptyset$ for some $\bm{p}\in\Z^n$ absolutely $q$--primitive and some $\bm{p_1}\in\Z^n$. Then 
\begin{equation}\label{majoq_1parrapaq}
\frac{1}{q q_1}\, \le \, \left|\frac{\bm{p}}{q} - \frac{\bm{p_1}}{q_1} \right| \, < \, \frac{2}{q_1^{\tau}}, \quad \textrm{ whence } \quad q_1^{\tau -1} < 2q.
\end{equation}
Furthermore, inequalities~(\ref{localisationpietp1i}) still hold true.

Given $\epsilon>0$ and $i\in \llbracket 1 , n \rrbracket$, it follows from Lemma~\ref{distrinbscopremiers} that the number of solutions in $p_i$ to~(\ref{localisationpietp1i}) is 
\begin{equation}\label{estimfaible}
\varphi(n) \left(\frac{4}{q_1^{\tau}} + o\left(\frac{1}{q^{1-\epsilon}}\right)\right) \, \underset{(\ref{majoq_1parrapaq})}{\le} \, \frac{6\varphi\left(q\right)}{q_1^{(\tau -1)(1-\epsilon)}}
\end{equation} for $q_0$ (and so $q_1$ and $q$) large enough depending on the choice of $\epsilon>0$ (note that the error term in Lemma~\ref{distrinbscopremiers} is independent of $\eta>0$).
Now, if there is an overlap between $C_{\tau}\!\!\left(\frac{\bm{p_1}}{q_1}\right)$ and $C_{\tau}\!\!\left(\frac{\bm{p_0}}{q_0}\right)$, it is easily seen that $p_{1,i}$ can assume at most $8q_1 \, \lambda\!\left(I_{\tau}\!\!\left(\frac{p_{0,i}}{q_0}\right)\right)$ values (where $\bm{p_0}=\left(p_{0,i}\right)_{1\le i \le n}$), so the number of solutions to~(\ref{localisationpietp1i}) in $\bm{p}\in\Z^n$ absolutely $q$--primitive is at most $$ 8^n \left(\frac{6\varphi\left(q\right)}{q_1^{(\tau -1)(1-\epsilon)-1}}\right)^n \, \lambda_n\!\left(C_{\tau}\!\!\left(\frac{\bm{p_{0}}}{q_0}\right)\right)$$ for $q_0$ large enough.

Summing over all the possible values for $q_1$, the number of hypercubes $C_{\tau}\!\!\left(\frac{\bm{p}}{q}\right)$ with $\bm{p}\in\Z^n$ absolutely $q$--primitive having a non--empty intersection with a hypercube of the form $C_{\tau}\!\!\left(\frac{\bm{p_1}}{q_1}\right)$ is seen to be less that 
\begin{align}\label{comptagesecondcas}
& 48^n \varphi(q)^n \, \lambda_n\!\left(C_{\tau}\!\!\left(\frac{\bm{p_{0}}}{q_0}\right)\right) \, \sum_{q_0^{\tau}/4< q_1 \le (2q)^{1/(\tau -1)}} q_1^{-n\left((\tau -1)(1-\epsilon)-1\right)} \nonumber \\
& \quad \le \, 48^n \varphi(q)^n \, \lambda_n\!\left(C_{\tau}\!\!\left(\frac{\bm{p_{0}}}{q_0}\right)\right) \, \sum_{q_1 > q_0^{\tau}/4} q_1^{-n\left((\tau -1)(1-\epsilon)-1\right)} \, \le \, \frac{c_2 \varphi(q)^n \, \lambda_n\!\left(C_{\tau}\!\!\left(\frac{\bm{p_{0}}}{q_0}\right)\right)}{q_0^{\tau\left(n\left((\tau -1)(1-\epsilon)-1\right) -1 \right)}}
\end{align}
for $q_0$ large enough depending on the choice of an arbitrarily small $\epsilon>0$ and for some $c_2>0$ depending on $\tau$ and $n$.

\paragraph{}\emph{Conclusion.} Taking into account~(\ref{comptagesansrestrictionabspprimitive}), (\ref{comptagepremiercas}) and (\ref{comptagesecondcas}), for $q>q_0^{\tau^3}$ large enough, 
$$\mathcal{N}\left(q,\frac{\bm{p_0}}{q_0}, \tau\right) \, \ge \, \frac{\varphi(q)^n}{2^n} \, \lambda_n\!\left(C_{\tau}\!\!\left(\frac{\bm{p_{0}}}{q_0}\right)\right) \left[1- 2^n\frac{q_0^{n\tau}}{2^n} \frac{c_1}{q_0^{(\tau-1)(n\tau -1)}} - \frac{2^n c_2}{q_0^{\tau\left(n\left((\tau -1)(1-\epsilon)-1\right) -1 \right)}} \right]$$ (we used the fact that $\lambda_n\!\left(C_{\tau}\!\!\left(\frac{\bm{p_{0}}}{q_0}\right)\right) = 2^n/q_0^{n\tau}$). This holds provided that $q_0$ satisfies the assumptions of~(\ref{comptagesansrestrictionabspprimitive}), (\ref{comptagepremiercas}) and~(\ref{comptagesecondcas}).

Now if $\epsilon>0$ has been chosen small enough, this last quantity is bigger than $\varphi(q)^n\, \lambda_n\!\left(C_{\tau}\!\!\left(\frac{\bm{p_{0}}}{q_0}\right)\right) / 2^{n+1}$ for $q_0$ large enough if $\tau > 1+ (1+1/n)/(1-\epsilon)$. The result follows on letting $\epsilon$ tend to zero.
\end{proof}

Proposition~\ref{originedelacontrainte} imposes the constraint $\tau>2+1/n$ in the statement of Theorem~\ref{thmprinciliminf}. The nature of this constraint appears to be twofold~: on the one hand, one could expect to improve inequalities~(\ref{comptagepremiercas}) by restricting the summation over only those integers $q_1$ for which there exists, in the first case of the proof, an overlap between $C_{\tau}\!\!\left(\frac{\bm{p_0}}{q_0}\right)$ and $C_{\tau}\!\!\left(\frac{\bm{p_1}}{q_1}\right)$ for some $\bm{p_1}\in\Z^n$. On the other hand, in the second case of the proof, Lemma~\ref{distrinbscopremiers} does not give enough information about the distribution of integers coprime to $q$ in very short intervals, so that estimate~(\ref{estimfaible}) leads to some loss of accuracy.

It is not clear however whether improvements on these inequalities will extend the result of Theorem~\ref{thmprinciliminf} to the case where $\tau$ lies in the interval $\left( 1+\nu\left(\mathcal{Q}\right)/n \, , \, 2+1/n \right)$. Indeed, one could also expect the Hausdorff dimension of liminf sets such as those under consideration to admit a ``phase transition'' at the critical value $\tau = 2+1/n$, that is to say the value of this dimension will be given by different expressions depending on whether $\tau$ is bigger or smaller than $2+1/n$. Such a phenomenon has already been conjectured in other situations --- see for instance Conjecture~1 of~\cite{bugeauddurand}.

In any case, restricting to the case $n=1$ for simplicity, the main underlying difficulty with the proof of Theorem~\ref{thmprinciliminf} turns out to be the control of the intersections of the intervals $I_{\tau}\!\!\left(\frac{p}{q}\right)$ and $I_{\tau}\!\!\left(\frac{p_1}{q_1}\right)$. This is also the notorious issue in proving the Duffin--Schaeffer conjecture~: as pointed out (and explained in more detail) in~\cite{dufschaefferextradivII}, this happens not just to be a deficiency in our knowledge but a \emph{real} problem in the sense that the intersection $I_{\tau}\!\!\left(\frac{p}{q}\right) \cap I_{\tau}\!\!\left(\frac{p_1}{q_1}\right)$ may be empty or it may well have a measure much bigger than the expected value $\lambda\!\left(I_{\tau}\!\!\left(\frac{p}{q}\right)\right) \times \lambda\!\left(I_{\tau}\!\!\left(\frac{p_1}{q_1}\right)\right)$ depending on the values taken by $p/q$ and $p_1/q_1$. It is likely that any further improvement on the bound for $\tau > 1+\nu\left(\mathcal{Q}\right)/n$ in Theorem~\ref{thmprinciliminf} would require the use of ideas very closely related to the problem of Duffin and Schaeffer.

\paragraph{} The last result of this subsection contains the main feature of the proof of Theorem~\ref{thmprinciliminf} and should be compared with Lemma~4 of~\cite{borofraen}.

\begin{lem}\label{lemmeconstruction}
Let $\tau > 2+1/n$, $\bm{p_0}\in\Z^n$ and $q_0\in\N$ such that $C_{\tau}\!\!\left(\frac{\bm{p_{0}}}{q_0}\right) \subset (0,1)^n$ and such that Proposition~\ref{originedelacontrainte} applies. Assume furthermore that $\nu\left(\mathcal{S}\right)>0$ and that $\delta_{2k}\left(\mathcal{S}\right) - \delta_{k}\left(\mathcal{S}\right) = o\left(\frac{k^{\nu\left(\mathcal{S}\right)}}{\left(\log \log k\right)^n} \right).$

Then for any $k>q_0$ sufficiently large, there exists a set $\mathcal{E}_{\tau}\!\!\left(\frac{\bm{p_0}}{q_0}\right)$ of rational vectors contained in $C_{\tau}\!\!\left(\frac{\bm{p_0}}{q_0}\right)$ such that~:
\begin{itemize}
\item[i)] for any $\frac{\bm{p}}{q}\in \mathcal{E}_{\tau}\!\!\left(\frac{\bm{p_0}}{q_0}\right)$, $\bm{p}\in\Z^n$ is absolutely $q$--primitive, $q\in\mathcal{S}$ and $k<q\le 2k$ ;
\item[ii)] for any two distinct elements $\frac{\bm{p_1}}{q_1}$ and $\frac{\bm{p_2}}{q_2}$ in $\mathcal{E}_{\tau}\!\!\left(\frac{\bm{p_0}}{q_0}\right)$ such that $q_1\le q_2$, $$\left|\frac{\bm{p_1}}{q_1}-\frac{\bm{p_2}}{q_2}\right|\, \ge \, \frac{1}{q_1^{1+\nu\left(\mathcal{S}\right)/n}} \; ;$$
\item[iii)] for any $\frac{\bm{p}}{q}\in \mathcal{E}_{\tau}\!\!\left(\frac{\bm{p_0}}{q_0}\right)$, $C_{\tau}\!\!\left(\frac{\bm{p}}{q}\right)$ is a hypercube of new generation in $C_{\tau}\!\!\left(\frac{\bm{p_{0}}}{q_0}\right)$ ;
\item[iv)] the following holds true~: $$\# \mathcal{E}_{\tau}\!\!\left(\frac{\bm{p_0}}{q_0}\right) \, \ge \, \frac{\lambda_n\!\left(C_{\tau}\!\!\left(\frac{\bm{p_{0}}}{q_0}\right)\right)}{2^{n+2}}\, \sum_{\underset{q\in\mathcal{S}}{k<q\le 2k}} \varphi(q)^n\, \gg \, \lambda_n\!\left(C_{\tau}\!\!\left(\frac{\bm{p_{0}}}{q_0}\right)\right) \frac{k^n \left(\delta_{2k}\left(\mathcal{S}\right) - \delta_{k}\left(\mathcal{S}\right)\right)}{\left(\log \log k\right)^n},$$ where the implicit constant depends only on $n$.
\end{itemize}
\end{lem}

\begin{proof}
For the sake of simplicity, let $C:=C_{\tau}\!\!\left(\frac{\bm{p_0}}{q_0}\right)$ in this proof only. Denote by $\mathcal{F}\left(C\right)$ the set of rational vectors $\frac{\bm{p}}{q}$ such that~:
\begin{itemize}
\item[1)] $q\in\mathcal{S}$ and $k< q \le 2k$;
\item[2)] $\bm{p}\in\Z^n$ is absolutely $q$--primitive;
\item[3)] $C_{\tau}\!\!\left(\frac{\bm{p}}{q}\right)$ is a hypercube of new generation in $C$.
\end{itemize}
If $\frac{\bm{p}}{q}$ and $\frac{\bm{p'}}{q'}$ are two rational vectors satisfying \emph{1)} and \emph{2)}, and if furthermore  $C_{\tau}\!\!\left(\frac{\bm{p}}{q}\right) \cap C_{\tau}\!\!\left(\frac{\bm{p'}}{q'}\right) \neq \emptyset$, then $$\frac{1}{4k^2} \, \le \, \left|\frac{\bm{p}}{q} - \frac{\bm{p'}}{q'}\right| \, \le \, \frac{1}{k^{\tau}},$$ which cannot happen if $\tau > 2$ and $k>q_0$ is chosen large enough. This shows together with Proposition~\ref{originedelacontrainte} that for such an integer $k$, 
\begin{equation}\label{rationnvellegenebis}
\# \mathcal{F}\left(C\right) \, \ge \, \frac{\lambda_n\!\left(C\right)}{2^{n+1}} \sum_{\underset{q\in\mathcal{S}}{k<q\le 2k}} \varphi(q)^n.
\end{equation}
Let $\mathcal{E}\left(C\right)$ be the subset of $\mathcal{F}\left(C\right)$ from which one excludes all the rational vectors $\frac{\bm{p}}{q}$ for which there exists an integer $q_1\in \llbracket k+1 , q-1 \rrbracket$ and an element $\frac{\bm{p_1}}{q_1} \in \mathcal{F}\left(C\right)$ satisfying $$\left|\frac{\bm{p_1}}{q_1} - \frac{\bm{p}}{q} \right|\, < \, \frac{1}{\left(q_1\right)^{1+\nu\left(\mathcal{S}\right)/n}}\cdotp$$ It should be clear that $\mathcal{E}\left(C\right)$ defined this way satisfies the conclusions \emph{i)} to  \emph{iii)} of the lemma. It remains to evaluate its cardinality.

Let $q_1, q \in \mathcal{S}$, $k< q_1 < q \le 2k$. When $q$ is fixed, denote by $N_i\left(q,q_1\right)$ ($1\le i \le n$) the number of integers $p_i$ such that there exists an integer $p_{1,i}$ satisfying 
\begin{equation}\label{borneentiersencore}
\left|p_{1,i}q - p_i q_1 \right| \, < \, \frac{q}{q_1^{\nu\left(\mathcal{S}\right)/n}}\cdotp
\end{equation}
Let furthermore $N(q)$ be the number of elements in $\mathcal{F}\left(C\right)\backslash \mathcal{E}\left(C\right)$~: it should be clear that  $$N(q) \, \le \, \sum_{\underset{q\in\mathcal{S}}{k<q_1<q}}\prod_{i=1}^{n} N_i\left(q, q_1\right).$$
From a familiar argument in elementary number theory (see for instance Lemma~I in~\cite{dufsch}), the number of solutions $N_i\left(q,q_1\right)$ in $p_i$ to~(\ref{borneentiersencore}) is bounded above by $2q/q_1^{\nu\left(\mathcal{S}\right)/n}$, whence $$N(q) \, \le \, 2^n q^n \sum_{\underset{q\in\mathcal{S}}{k<q_1<q}}\frac{1}{(q_1)^{\nu\left(\mathcal{S}\right)}}\, \le \, 2^n q^n \frac{\delta_{2k}\left(\mathcal{S}\right) - \delta_{k}\left(\mathcal{S}\right)}{k^{\nu\left(\mathcal{S}\right)}}\cdotp$$
Using the well--known result
\begin{equation}\label{ordreinffcteuler}
\liminf_{m\rightarrow +\infty}\left(\frac{\varphi(m)\log \log m}{m}\right) = e^{-\gamma},
\end{equation}
where $\gamma$ is Euler--Mascheroni constant, this also leads to the estimate valid for $k$ large enough $$N(q) \, \le \, 2^n e^{\gamma n} 2^n \varphi\left(q\right)^n \left(\log \log k\right)^n \frac{\delta_{2k}\left(\mathcal{S}\right) - \delta_{k}\left(\mathcal{S}\right)}{k^{\nu\left(\mathcal{S}\right)}},$$
where the implicit constant is absolute.

Now, by assumption on the sequence $\left(\delta_{2k}\left(\mathcal{S}\right) - \delta_{k}\left(\mathcal{S}\right)\right)_{k\ge 1}$, for $k$ large enough, $2^n e^{\gamma n} \left(\log \log k\right)^n \left(\delta_{2k}\left(\mathcal{S}\right) - \delta_{k}\left(\mathcal{S}\right)\right) k^{-\nu\left(\mathcal{S}\right)}\, \le \, \lambda_n\!\left(C\right) / 2^{2n+2}$, so that
\begin{equation}\label{nqbonnestim}
N(q) \, \le \, \frac{\varphi(q)^n \lambda_n\!\left(C\right)}{2^{n+2}}\cdotp
\end{equation}
For such an integer $k$, from~(\ref{rationnvellegenebis}) and~(\ref{nqbonnestim}), 
\begin{align*}
\# \mathcal{E}\left(C\right) = \# \mathcal{F}\left(C\right) - \# \left(\mathcal{F}\left(C\right)\backslash \mathcal{E}\left(C\right)\right) \, &\ge \, \sum_{\underset{q\in\mathcal{S}}{k< q \le 2k}} \left(\frac{\varphi(q)^n \lambda_n\!\left(C\right)}{2^{n+1}} - N(q) \right) \\
& \ge \, \frac{\lambda_n\!\left(C\right)}{2^{n+2}} \sum_{\underset{q\in\mathcal{S}}{k< q \le 2k}} \varphi(q)^n \\
& \underset{(\ref{ordreinffcteuler})}{\gg} \, \lambda_n\!\left(C\right) \frac{k^n}{(\log \log k)^n} \left(\delta_{2k}\left(\mathcal{S}\right) - \delta_{k}\left(\mathcal{S}\right)\right).  
\end{align*}
\end{proof}

\section{Proof of the main Theorem}

Theorem~\ref{thmprinciliminf} will now be proved for a given infinite set of positive integers $\mathcal{Q}$. As should be clear, it is enough to establish the result for the set $W^*_{\tau, n}(\mathcal{Q}) \cap [0,1]^n$.

\subsection{The upper bound}

Proving that $\dim W^*_{\tau, n}(\mathcal{Q}) \le \left(1+\nu(\mathcal{Q})\right)/\tau$ is almost trivial~: for any $N\ge 1$, $$\bigcup_{\underset{q\in\mathcal{Q}}{q\ge N}} \bigcup_{\bm{p}\in\llbracket 0, q \rrbracket ^n} C_{\tau}\!\!\left(\frac{\bm{p}}{q}\right)$$ is a cover of the limsup set $W_{\tau, n}(\mathcal{Q})$, so in particular of the liminf set $W^*_{\tau, n}(\mathcal{Q})$. Consequently, for any $N\ge 1$, the $s$--dimensional Hausdorff measure $\mathcal{H}^s\!\left(W^*_{\tau, n}(\mathcal{Q})\right)$ of the set $W^*_{\tau, n}(\mathcal{Q})$ satisfies $$\mathcal{H}^s\!\left(W^*_{\tau, n}(\mathcal{Q})\right) \, \le \, \sum_{\underset{q\in\mathcal{Q}}{q\ge N}} \frac{(q+1)^n}{q^{s\tau}}\cdotp$$ The right--hand side of this inequality is finite as soon as $s>(n+\nu\left(\mathcal{Q}\right))/\tau$, hence $\dim W^*_{\tau, n}(\mathcal{Q}) \le (n+\nu\left(\mathcal{Q}\right))/\tau$ for any $\tau>1+\nu\left(\mathcal{Q}\right)/n$.

\subsection{The lower bound}

The core of the proof of Theorem~\ref{thmprinciliminf} consists of establishing the correct lower bound for $\dim W^*_{\tau, n}(\mathcal{Q})$. The ideas developed here  are inspired by Chapters~1 and~4 of~\cite{falcofractalgeometry} and by~\cite{borofraen} (which is based itself on the pioneer work of Jarn\'ik~\cite{multijar}). 

\paragraph{} Recall first the construction of a \emph{level set} $E$ in $[0,1]^n$~: let $$[0,1]^n= E_0 \supset E_1 \supset E_2 \supset \dots$$ be a decreasing sequence of sets such that each $E_k$ is a finite union of disjoint and closed hypercubes. Assume furthermore that each hypercube of $E_k$ contains $m_k\ge 2$ hypercubes from $E_{k+1}$ and that the maximal diameter of the hypercubes of level $k$ (i.e. in $E_k$) tends to 0 as $k$ goes to infinity. Then 
\begin{equation}\label{ensdeniveau}
E:= \bigcap_{k=0}^{+\infty}E_k
\end{equation} 
is a totally disconnected subset of $[0,1]^n$ --- a Cantor set --- referred to as a level set.

It is possible to equip such a level set $E$ with a measure $\mu $ supported on it in the following way~: let $\mu_0$ be the uniform distribution on $E_0=[0,1]^n$. If $\mu_{k-1}$ is a measure supported on $E_{k-1}$ previously defined, let $\mu_k$ be the measure supported on $E_k$ assigning a mass of $\left(m_1\dots m_k\right)^{-1}$ to each of the $m_1\dots m_k$ hypercubes of $E_k$, the distribution of $\mu_k$ on each of these hypercubes being uniform. Denote by $\mathcal{E}$  the set of hypercubes of all levels used to construct $E$. For any $U\in\mathcal{E}$ of level $k$, let $\mu\left(U\right) := \mu_k\left(U\right) = \left(m_1\dots m_k\right)^{-1}$. If one sets, for any $A\subset \R^n$, 
\begin{equation}\label{mesureenscantor}
\mu\left(A\right):= \inf \left\{ \sum_{l=0}^{+\infty}\mu\left(U_l\right)\; : \; A\cap E \subset \bigcup_{l=0}^{+\infty} U_l \textrm{   and   } U_l\in\mathcal{E}\right\},
\end{equation}
then $\mu$ defines a probability measure supported on $E$ (see chapter 1 of~\cite{falcofractalgeometry} for details).

Such a measure often turns out to be useful when establishing a lower bound for $\dim E$ by virtue of the well--known Mass Distribution Principle which is now recalled (cf. for instance~\cite{falcofractalgeometry} for a proof).

\begin{thm}[Mass Distribution Principle]\label{massdistributionprinciple}
Let $E$ be a level set as described above supporting a probability measure $\mu$. Assume furthermore that for some $s\ge 0$, there exist numbers $c, \kappa>0$ such that 
\begin{equation}\label{majomesureprincidistribmass}
\mu\left(U\right) \, \le \, c \left|U\right|^s
\end{equation} 
for all hypercubes $U\in\R^n$ satisfying $\left|U\right|\le \kappa$ (recall that $\left|U\right|$ denotes the diameter of $U$). 

Then $$\dim E \ge s.$$
\end{thm}

This principle shall now be used to compute the Hausdorff dimension of sufficiently large level sets contained in $W^*_{\tau, n}(\mathcal{Q})$.

\subsubsection{The case $\nu\!\left(\mathcal{Q}\right)>0$}

Assume first that $\nu\left(\mathcal{Q}\right)>0$ and let $\delta\in \left(0, \nu\left(\mathcal{Q}\right)/2\right)$. 

Since the sequence $\left(n^{\nu\left(\mathcal{Q}\right)-\delta}/\log n\right)_{n\ge 2}$ is increasing for $n$ large enough, Proposition~\ref{sousensbiendistribue} guarantees the existence of a subset $\mathcal{Q}_{\delta} \subset \mathcal{Q}$ for which one can find a strictly increasing sequence of natural integers $\left(n_k\right)_{k\ge 1}$ satisfying $$\delta_{2n_k}\left(\mathcal{Q}_{\delta}\right) - \delta_{n_k}\left(\mathcal{Q}_{\delta}\right)\, \underset{k\rightarrow +\infty}{\sim} \, \frac{n_k^{\nu\left(\mathcal{Q}_{\delta}\right)}}{\log n_k},$$ where $\nu\!\left(\mathcal{Q}_{\delta}\right) = \nu\!\left(\mathcal{Q}\right)-\delta$.

In the general construction of a level set, let $E_0:=[0,1]^n$ and, for $q_1\in\mathcal{Q}_{\delta}$, $q_1\ge 2$, $$E_1 = \bigcup_{\bm{p_1}\in\llbracket 1, q_1 -1 \rrbracket ^n} C_{\tau}\!\!\left(\frac{\bm{p_1}}{q_1}\right).$$ If $E_{k-1}$ ($k\ge 2$) has been defined, let $C_{\tau}\!\!\left(\frac{\bm{p_{k-1}}}{q_{k-1}}\right)$ be one of its connected components contained in $(0,1)^n$. From Lemma~\ref{lemmeconstruction}, there exists an element $q_k > q_{k-1}$ in the sequence $\left(n_k\right)_{k\ge 1}$ and 
\begin{equation}\label{nbrhypercubes}
m_k \, \gg \, \lambda_n\!\left(C_{\tau}\!\!\left(\frac{\bm{p_{k-1}}}{q_{k-1}}\right)\right)\, \sum_{\underset{q\in\mathcal{S}}{q_k<q\le 2q_k}} \varphi(q)^n\, \gg \, \frac{q_k^{n+\nu\!\left(\mathcal{Q}_{\delta}\right)}}{q_{k-1}^{n\tau}\left(\log q_k\right)\left(\log \log q_k\right)^n}
\end{equation}
hypercubes of new generation in $C_{\tau}\!\!\left(\frac{\bm{p_{k-1}}}{q_{k-1}}\right)$ of the form $C_{\tau}\!\!\left(\frac{\bm{p}}{q}\right)$ with $q_k < q \le 2q_k$ and $q\in\mathcal{Q}_{\delta}.$ Furthermore, the distance between these hypercubes is at least 
\begin{equation}\label{ecarthypercubes}
\epsilon_k := \frac{1}{2\left(q_k\right)^{1+\nu\!\left(\mathcal{Q}_{\delta}\right)/n}}
\end{equation}
(by convention, $\epsilon_0:=1$).

Let then $E_k$ be defined as the union of all these hypercubes over all the connected components of $E_{k-1}$ and let $E$ be as in~(\ref{ensdeniveau}). By construction, $E\subset W^*_{\tau, n}\left(\mathcal{Q}_{\delta}\right) \subset W^*_{\tau, n}(\mathcal{Q})$ and $E$ supports a probability measure $\mu$ as mentioned in~(\ref{mesureenscantor}).

\begin{rem}
The connected components of $E_k$ ($k\ge 1$) are of the form $C_{\tau}\!\!\left(\frac{\bm{p}}{q}\right)$ for some $\bm{p}\in\Z^n$ and $q\in\mathcal{Q}$ and so are not closed as in the definition of a level set. This difficulty can easily be overcome by redefining them as the closure of the same hypercubes whose side lengths are shrunk by a factor $1-\eta$ for some $\eta<1/2$. It is then readily checked that Proposition~\ref{originedelacontrainte} and Lemma~\ref{lemmeconstruction} remain true up to an additional multiplicative constant which shall not cause any trouble at all in the rest of the proof. For the sake of simplicity of notation, such detail shall be omitted in what follows.
\end{rem}

\paragraph{} Letting $$\rho:= \frac{n+\nu\left(\mathcal{Q}_{\delta}\right)-\delta}{\tau} = \frac{n+\nu\left(\mathcal{Q}\right)-2\delta}{\tau},$$ it will now be shown by induction on $k\ge 0$ that the sequence $\left(q_k\right)_{k\ge 0}$ may be chosen in such a way that, for any hypercube $U\subset \R^n$, (\ref{majomesureprincidistribmass}) holds true with $s= \rho$ for some real $c>0$ to be defined later. The following simplifies a great deal the method of~\cite{borofraen}.

Let $U$ be a hypercube in $\R^n$ and let $k\ge 0$ be such that $\epsilon_{k+1} \le \left|U\right| < \epsilon_k$ (this comes down to taking $\kappa = \epsilon_0 = 1$ in Theorem~\ref{massdistributionprinciple}). Then $U$ intersects at most one connected component of $E_k$ and, since the measure $\mu$ is supported on $E$, there is no loss of generality in assuming that it is actually contained in this connected component. Furthermore, it may also be assumed that $U$ intersects $E_{k+1}$ (otherwise $\mu\left(U\right) = 0$ again from~(\ref{mesureenscantor}) and the result to prove is trivial). Thus, under these conditions, it follows from~(\ref{mesureenscantor}) that  $$\mu\left(U\right)\, \le \, \mu_{k+1}\left(U\right),$$ where $\mu_{k+1}$ is the uniform distribution supported by $E_{k+1}$.

All this shows that it is enough to prove by induction on $k\ge 0$ the following statement~:\\
$(H_k)$~: \emph{For any hypercube} $U$ \emph{contained in a connected component of} $E_k,$ \emph{having a non--empty intersection with} $E_{k+1}$ \emph{and satisfying furthermore} $\epsilon_{k+1} \le \left|U\right| < \epsilon_k ,$ 
  $$\frac{\mu_{k+1}\left(U\right)}{\left|U \right|^{\rho}}\,\le \, c.$$

Note that for any hypercube $U\subset [0,1]^n$, $$\frac{\mu_{0}\left(U\right)}{\left|U \right|^{\rho}} = \frac{\lambda_{n}\!\left(U\right)}{\left|U \right|^{\rho}}\,< \, \left|U\right|^{n-\rho}\, \le \, 1.$$ Therefore, it shall be assumed that $c\ge 1$.

\paragraph{} Consider now an integer $k \ge 0$ and a hypercube $U$ satisfying the assumptions of $(H_k)$. Let $C_k$ be the connected component of $E_k$ containing $U$ and let $N_{U}$ denote the number of connected components of $E_{k+1}$ having a non--empty intersection with $U$. By assumption, $N_{U}\ge 1$. The conclusion of $(H_k)$ is proved by distinguishing two subcases.

\paragraph{}\emph{First subcase~:} $\left|U \right| \ge \left(q_{k+1}\right)^{-1/2}.$ Under this assumption, if $q_1$ is chosen large enough so that Lemma~\ref{distrinbscopremiers} applies with $\epsilon = 1/2$, then, for all $k\ge 0$, 
\begin{equation}\label{nbrdintersectiondeU}
N_U \, \ll \, \left|U\right|^n  \sum_{\underset{q\in\mathcal{S}}{q_{k+1}< q \le 2q_{k+1}}} \varphi(q)^n,
\end{equation}
hence 
\begin{align*}
\frac{\mu_{k+1}\left(U\right)}{\left|U \right|^{\rho}} \, &\le \, \frac{\mu_{k+1}\left(C_k\right)}{\left|U \right|^{\rho}} = \frac{\mu_{k}\left(C_k\right)}{m_{k+1}} \frac{1}{\left|U \right|^{\rho}} \, \le \, \frac{\mu_{k}\left(C_k\right)}{m_{k+1}} \frac{N_U}{\left|U \right|^{\rho}} \\
& \underset{(\ref{nbrhypercubes})\, \& \,(\ref{nbrdintersectiondeU})}{\ll} \, \frac{\mu_{k}\left(C_k\right)}{\left|U \right|^{\rho}}\frac{\left|U \right|^{n}}{\left|C_k \right|^{n}} \left(\sum_{\underset{q\in\mathcal{S}}{q_{k+1}< q \le 2q_{k+1}}} \varphi(q)^n\right) \left(\sum_{\underset{q\in\mathcal{S}}{q_{k+1}< q \le 2q_{k+1}}} \varphi(q)^n\right)^{-1} \\
& \qquad  \qquad  = \frac{\mu_{k}\left(C_k\right)}{\left|C_k \right|^{\rho}} \left(\frac{\left|U \right|}{\left|C_k \right|}\right)^{n-\rho} \\
& \le \, \frac{\mu_{k}\left(C_k\right)}{\left|C_k \right|^{\rho}}\cdotp
\end{align*}

If $k=0$, this means that there exists a constant $K\ge 1$ such that $$\frac{\mu_{1}\left(U\right)}{\left|U \right|^{\rho}} \, \le \, K\, \frac{\mu_{0}\left(C_0\right)}{\left|C_0 \right|^{\rho}},$$ where $C_0=[0,1]^n$. Choosing $c$ bigger than this last quantity proves the result in this case.

If $k\ge 1$, then, denoting by $C_{k-1}$ the connected component of $E_{k-1}$ containing $C_k$, $$\frac{\mu_{k}\left(C_k\right)}{\left|C_k \right|^{\rho}} = \frac{\mu_{k-1}\left(C_{k-1}\right)}{m_k\left|C_k \right|^{\rho}} \, \ll \, \frac{\mu_{k-1}\left(C_{k-1}\right) q_{k-1}^{n\tau}\left(\log q_k\right) \left(\log \log q_k\right)^n}{q_k^{\delta}},$$ the last inequality following from~(\ref{nbrhypercubes}) and the fact that $\left|C_k\right| \asymp q_k^{-\tau}.$ Choosing $q_k$ large enough in the previous step, this quantity can be made arbitrarily small.

\paragraph{}\emph{Second subcase~:} $\left|U \right| \le \left(q_{k+1}\right)^{-1/2}.$ By assumption, $\epsilon_{k+1} \le \left|U\right|$. Since two connected components of $E_{k+1}$ are distant from at least $\epsilon_{k+1}$, inequality~(\ref{ecarthypercubes}) implies $$N_U \, \ll \, \left|U\right|^n \left(q_{k+1}\right)^{n+\nu\!\left(\mathcal{Q}_{\delta}\right)}.$$
Therefore, denoting by $C_{k+1}$ any connected component of $E_{k+1}$, 
\begin{align*}
\frac{\mu_{k+1}\left(U\right)}{\left|U \right|^{\rho}} \, &\le \, \frac{\mu_{k+1}\left(C_{k+1}\right)N_U}{\left|U \right|^{\rho}} \, \ll \, \frac{\mu_{k}\left(C_k\right)}{m_{k+1}} q_{k+1}^{n+\nu\!\left(\mathcal{Q}_{\delta}\right)} q_{k+1}^{-(n-\rho)/2} \\
& \underset{(\ref{nbrhypercubes})}{\ll} \, \frac{\mu_{k}\left(C_k\right)q_k^{n\tau}\left(\log q_{k+1}\right)\left(\log \log q_{k+1}\right)^n}{q_{k+1}^{(n-\rho)/2}}
\end{align*}
(for the second inequality, we used the fact that $C_{k+1} \subset C_k$). Choosing $q_{k+1}$ large enough, this quantity can be made arbitrarily small.

\paragraph{}\emph{Conclusion~:} From the Mass Distribution Principle (Theorem~\ref{massdistributionprinciple}), for any $\delta \in \left(0, \nu\!\left(\mathcal{Q}\right)/2\right)$, $$\dim W^*_{\tau, n}(\mathcal{Q}) \, \ge \, \dim E \, \ge \, \frac{n+\nu\!\left(\mathcal{Q}\right) - 2\delta}{\tau}\cdotp$$ Letting $\delta$ tend to zero completes the proof of Theorem~\ref{thmprinciliminf} in the case $\nu\!\left(\mathcal{Q}\right) > 0$.

\subsubsection{The case $\nu\!\left(\mathcal{Q}\right)=0$}

The proof in the case $\nu\!\left(\mathcal{Q}\right) = 0$ is a simplified version of the previous one. We only mention here the changes to make in the latter~: in the construction of the level set $E$, assume that $E_{k-1}$ ($k\ge 2$) has been defined and let $C_{\tau}\!\!\left(\frac{\bm{p_{k-1}}}{q_{k-1}}\right)$ be one of its connected components. For $q_k > q_{k-1}$ large enough, $q_k\in\mathcal{Q}$, Proposition~\ref{originedelacontrainte} guarantees the existence of at least $$m_k \, \gg \, \varphi\left(q_k\right)^n \, \lambda_n\!\left(C_{\tau}\!\!\left(\frac{\bm{p_{k-1}}}{q_{k-1}}\right)\right) \, \underset{(\ref{ordreinffcteuler})}{\gg} \, \frac{q_k^n}{q_{k-1}^{n\tau} \left(\log \log q_k\right)^n}$$ hypercubes of new generation in $C_{\tau}\!\!\left(\frac{\bm{p_{k-1}}}{q_{k-1}}\right)$ of the form $C_{\tau}\!\!\left(\frac{\bm{p}}{q_{k}}\right)$ ($\bm{p}\in\Z^n$) which are furthermore at least $$\epsilon_k := \frac{1}{2q_k^n}$$ apart (as should be obvious). The set $E_k$ is then defined as the union of all these hypercubes over all the connected components of $E_{k-1}$. 

The level set $E$ obtained this way may again be equipped with a probability measure supported on it. Given $\delta > 0$, the same argumentation as in the case $\nu\!\left(\mathcal{Q}\right) > 0$ shows that the sequence $\left(q_k\right)_{k\ge 0}$ may be chosen in such a way that the Mass Distribution Principle (Theorem~\ref{massdistributionprinciple}) leads to the estimate 
\begin{equation}\label{borneinfcasnuzero}
\dim E \, \ge \, \rho:=\frac{n-2\delta}{\tau}\cdotp
\end{equation} It should however be mentioned that inequality~(\ref{nbrdintersectiondeU}) must now be replaced by the following one~: $$N_U\, \ll \, \left|U\right|^n \varphi\left(q_{k+1}\right)^n.$$

Letting $\delta$ tend to zero in~(\ref{borneinfcasnuzero}) completes the proof of Theorem~\ref{thmprinciliminf} in this case also.

\section{A $p$--adic version of the main Theorem}

Let $p$ be an arbitrary but fixed prime. 

An analogue of Theorem~\ref{thmprinciliminf} is now studied in $\Q_p$. Consider first the $p$--adic version of the set of $\tau$--well approximable numbers ($\tau >0$) in $\Q_p$, namely
\begin{equation}\label{limsuppadique}
W_{\tau, n}(p) := \left\{ \bm{x}\in \Q_p^n \; : \; \left|q\bm{x}-\bm{r}\right|_p < \max\left(\left|\bm{r}\right|, q \right)^{-\tau} \; \; \mbox{ for  i.m. } (\bm{r},q) \in \Z^n\times\N \right\}.
\end{equation} 
Here, $\left| \bm{x} \right|_p$ denotes the supremum of the $p$--adic norms of the components of $\bm{x}\in\Q^n_p$.
 
Note that, unlike in~(\ref{enslimsup}), the approximating function depends now both on $\left|\bm{r}\right|$ and $q$ rather than simply $q$. This is due to the fact that, in the $p$--adic setup, given $x\in\Z_p$, a quantity of the form $\left|qx-r\right|_p$ can be made arbitrarily small by taking $r$ to be a rational integer with the appropriate number of leading terms taken from the $p$--adic expansion of $qx$. Thus the set of $\bm{x}\in\Q_p^n$ such that $\left|q\bm{x}-\bm{r}\right|_p < q^{-\tau}$ for infinitely many $(\bm{r},q) \in \Z^n\times\N$ contains the whole of $\Z_p^n$ and has therefore full Hausdorff dimension regardless of the value of $\tau>0$.

Another difference with~(\ref{enslimsup}) is that, in the $p$--adic setup, there is no ``normalizing'' factor $q$ on the right--hand side of $\left|q\bm{x}-\bm{r}\right|_p$. This is due to the fact that the $p$--adic norm is an ultra metric. For more details, the limsup set $W_{\tau, n}(p)$ is studied in full generality in~\cite{meastheorlawlimsup}.

Let $W_{\tau, n}^*(p)$ be the liminf set obtained from~(\ref{limsuppadique}) by imposing the constraint that all the integers $q$ should be divisible by $p$, namely
\begin{equation}\label{liminfpadique}
W_{\tau, n}^*(p) := \left\{ \bm{x}\in \Q_p^n \; : \;
\begin{split}
&\left|q\bm{x}-\bm{r}\right|_p < \max\left(\left|\bm{r}\right|, q \right)^{-\tau} \; \; \mbox{     for  i.m. } (\bm{r},q) \in \Z^n\times p\N \\
&\mbox{and f.m. } (\bm{r},q) \not\in \Z^n\times p\N 
\end{split}
\right\},
\end{equation}
where \emph{f.m.} stands for \emph{finitely many}. The set $W_{\tau, n}^*(p)$ may be seen as an analogue of at least two different real liminf sets as introduced in~(\ref{limfinreel})~: on the one hand, it is defined as the set of elements in $\Q_p^n$ which are $\tau$--well approximable only by integer vectors $(\bm{r},q)$ such that $q$ is a multiple of the integer $p$ provided it is large enough. On the other, since the the gcd of two $p$--adic integers is the highest power of $p$ dividing both of them (it is defined up to an invertible element), $W_{\tau, n}^*(p)$ is also the set of all elements in $\Q_p^n$ $\tau$--well approximable only by integer vectors $(\bm{r},q)$ such that, provided it is large enough, $q$ is \emph{not} coprime to a given non unit $s\in\Z_p$.

The structure of the liminf set $W_{\tau, n}^*(p)$ exhibits very different behaviours depending on whether it is restricted to $\Z_p$ or not.

\begin{thm}\label{thmprincipalliminfpadique}
If $\tau > 1+1/n$, then $$\dim W_{\tau, n}^*(p) = \frac{n+1}{\tau}\cdotp$$ 
Furthermore, $W_{\tau, n}^*(p) \cap \Z^n_p = \emptyset$ as soon as $\tau \ge1$. 
\end{thm}

Thus the situation is quite original~: the liminf set $W_{\tau, n}^*(p)$ has the same Hausdorff dimension as the limsup set $W_{\tau, n}(p)$ when $\tau > 1+1/n$ (cf. p.82 of~\cite{meastheorlawlimsup}) but it contains no $p$--adic integers. This is in contrast with the fact that, when considering the limsup set $W_{\tau, n}(p)$ from a metric point of view, it generally suffices to study its intersection with $\Z_p^n$ as the space $\Q_p^n$ can be written as a countable union of translates of $\Z_p^n$.

The proof of Theorem~(\ref{thmprincipalliminfpadique}) rests on the following lemma which uses Definition~\ref{vecteurprimitif}.

\begin{lem}\label{lemmepprimitifcaspadique}
If $\tau \ge 1$, then the limsup set $W_{\tau, n}(p)$ is also the set $$W_{\tau, n}(p) = \left\{ \bm{x}\in \Q_p^n \; : \; \left|q\bm{x}-\bm{r}\right|_p < \max\left(\left|\bm{r}\right|, q \right)^{-\tau} \; \; \mbox{ for  i.m. $p$--primitive } (\bm{r},q) \in \Z^n\times\N \right\}.$$
\end{lem}

\begin{proof}
Given $\bm{x}\in W_{\tau, n}(p)$, let $\left(u_k:=\left(\bm{r_k}, q_k\right)\right)_{k\ge 1}$ be the sequence strictly increasing in $q_k$ of elements of $\Z^n\times\N$ satisfying 
\begin{equation}\label{relationdapproximationpadique}
\left|q_k\bm{x}-\bm{r_k}\right|_p < \max\left(\left|\bm{r_k}\right|, q_k \right)^{-\tau}.
\end{equation}
Note that if $k_0$ and $m$ are positive integers, $m u_{k_0}$ satisfies~(\ref{relationdapproximationpadique}) if, and only if, 
$$1 \, \underset{(\tau \ge 1)}{\le} \, \left| m \right|_p \left| m\right|^{\tau} \, < \, \left|q_{k_0}\bm{x} - \bm{r_{k_0}}\right|_p^{-1} \max \left(q_{k_0}, \left|\bm{r_{k_0}}\right|\right)^{-\tau}.$$
The first of these inequalities shows that $u_{k_0}$ is a multiple of a $p$--primitive vector $\tilde{u}_{k_0}$ and the second one proves that the number of multiples of $u_{k_0}$ satisfying~(\ref{relationdapproximationpadique}) is finite.
\end{proof}

\begin{coro}\label{corollaireintersectionentierspadiquesvide}
Assume that $\tau \ge 1$. 

Then $$W_{\tau, n}^*(p) \cap \Z_p^n = \emptyset .$$
\end{coro}

\begin{proof}
Let $\bm{x} = \left(x_1, \dots , x_n\right)\in W_{\tau, n}^*(p) \cap \Z_p^n$ and let $(\bm{r}, q)\in \Z^n\times \N$ be a vector of approximation of $\bm{x}$, i.e. a vector satisfying~(\ref{relationdapproximationpadique}). From Lemma~\ref{lemmepprimitifcaspadique}, $(\bm{r}, q)$ may be assumed to be $p$--primitive which, from the definition of the liminf set $W_{\tau, n}^*(p)$ and provided that $q$ is large enough, implies on the one hand that $p|q$ and on the other that $\left| r_{i_0}\right|_p = 1$ for some component $r_{i_0}\in\Z$ of the vector $\bm{r}:=\left(r_1, \dots, r_n\right)\in\Z^n$. In particular, 
\begin{equation}\label{deuxiemerelationapproximationpadique}
\left|qx_{i_0}-r_{i_0} \right|_p < \max\left(q, \left|r_{i_0}\right|\right)^{-\tau}.
\end{equation}
Now if $1=\left|r_{i_0}\right|_p > \left|qx_{i_0}\right|_p$, (\ref{deuxiemerelationapproximationpadique}) implies that $\left|qx_{i_0}-r_{i_0} \right|_p = \left|r_{i_0}\right|_p = 1 < \left|r_{i_0}\right|_p^{-\tau}$, which is impossible.
If $1=\left|r_{i_0}\right|_p < \left|qx_{i_0}\right|_p$, then it follows from~(\ref{deuxiemerelationapproximationpadique}) that $1\, < \, \left|qx_{i_0}\right|_p  = \left|qx_{i_0}-r_{i_0} \right|_p \, < \, q^{-\tau}$, which cannot happen.
Finally, if $\left|r_{i_0}\right|_p =1 = \left|qx_{i_0}\right|_p$, then, since $p|q$, $1\, > \, \left|q\right|_p = \left|x_{i_0}\right|_p^{-1} \, \underset{(x_0\in\Z_p)}{\ge} \, 1,$ which gives again a contradiction.
This completes the proof of the corollary.
\end{proof}

\paragraph{}\emph{Completion of the proof of Theorem~\ref{thmprincipalliminfpadique}.}
From the proof of Corollary~\ref{corollaireintersectionentierspadiquesvide}, it also follows that if $\left(\bm{r},q\right)$ is a $p$--primitive vector of approximation of $\bm{x} = \left(x_1, \cdots, x_n\right)\in W_{\tau, n}^*(p)$ such that $p|q$ and $p$ does not divide a component $r_{i_0}\in\Z$ of $\bm{r}$, then, necessarily, $\left|r_{i_0}\right|_p = 1 = \left|qx_{i_0}\right|_p$. This implies in particular that $x_{i_0}\in \Q_p\backslash\Z_p$. Note also that the condition $\left|qx_{i_0}\right|_p=1$ alone is sufficient to guarantee that $\left|r_{i_0}\right|_p = 1$~: indeed, if one had $\left|r_{i_0}\right|_p < 1 = \left|qx_{i_0}\right|_p$, then one would also have $\left|qx_{i_0}\right|_p = \left|qx_{i_0} -r_{i_0}\right|_p = 1 < q^{-\tau}$, which cannot be.

Thus, each $p$--primitive vector of approximation $\left(\bm{r},q\right)$ of $\bm{x}\in W_{\tau, n}^*(p)$ determines at least one component $x_{i_0}$ of $\bm{x}$ such that $x_{i_0}\in\Q_p \backslash \Z_p$. Since there are only finitely many components, it follows that
\begin{equation*}
W_{\tau, n}^*(p) = \left\{ \bm{x}\in \Q_p^n \; : \; \exists i_{0}\in\llbracket 1, n \rrbracket, \, x_{i_0}\in\Q_p \backslash \Z_p \mbox{ and }
\begin{split}
&\left|q\bm{x}-\bm{r}\right|_p < \max\left(\left|\bm{r}\right|, q \right)^{-\tau}\\
&\left|qx_{i_0}\right|_p = \left|r_{i_0}\right|_p = 1
\end{split}
\mbox{  i.o. }\right\},
\end{equation*}
where \emph{i.o.} stands for \emph{infinitely often}. Therefore,

\begin{equation*}
W_{\tau, n}^*(p) = \bigcup_{i_0=1}^{n}\left\{ \bm{x}\in \Q_p^n \; : \; x_{i_0}\in\Q_p \backslash \Z_p \mbox{ and }
\begin{split}
&\left| q\bm{x}-\bm{r}\right|_p < \max\left(\left|\bm{r}\right|, q \right)^{-\tau}\\
&\left| qx_{i_0}\right|_p = 1
\end{split}
\mbox{  i.o. }\right\}.
\end{equation*}
For any $f\in\N$, denote by $W_{\tau, n}\left(p, i_0, f\right)$  the set
\begin{equation*}
W_{\tau, n}\left(p, i_0, f\right) := \left\{ \bm{x}\in \Q_p^n \; : \; \left|x_{i_0}\right|_p = p^{f} \mbox{ and }
\begin{split}
&\left|q\bm{x}-\bm{r}\right|_p < \max\left(\left|\bm{r}\right|, q \right)^{-\tau}\\
&\left|qx_{i_0}\right|_p = 1
\end{split}
\mbox{  i.o. }\right\}.
\end{equation*}
Then
\begin{align*}
W_{\tau, n}^*(p) =  \bigcup_{i_0=1}^{n} \bigcup_{f=1}^{+\infty} W_{\tau, n}\left(p, i_0, f\right)
\end{align*}
and it suffices to establish the dimensional result in Theorem~\ref{thmprincipalliminfpadique} for any of the sets $W_{\tau, n}\left(p, i_0, f\right)$. 

Fix $f\ge 1$ and $i_0\in\llbracket 1, n \rrbracket$. Given $(\bm{r},q) \in \Z^n\times \N$, let $\nu(\bm{r},q):= \max\left(q, \left|\bm{r}\right| \right)$ and let $$B\left(\bm{x}, \rho \right):=\left\{\bm{a}\in\Q_p^n\; : \; \left|\bm{x}-\bm{a}\right|_p < \rho \right\}$$ denote the open ball of radius $\rho>0$ centered at $\bm{x}\in\Q_p^n$. It should then be clear that a cover for $W_{\tau, n}\left(p, i_0, f\right)$ is given by $$\bigcap_{N=1}^{+\infty} \bigcup_{\nu > N} \bigcup_{\nu(\bm{r},q)\in \mathcal{A}_{\nu}\left(i_0,f\right)}B\left(\frac{\bm{r}}{q}, \nu^{-\tau}p^f\right),$$ where $$\mathcal{A}_{\nu}\left(i_0,f\right) := \left\{(\bm{r},q) \in \Z^n\times \N \; : \; \left|q\right|_p = p^{-f}, \left|r_{i_0}\right|_p = 1 \mbox{ and } \nu(\bm{r},q)=\nu \right\}.$$
Furthermore, it is readily checked that $\# \mathcal{A}_{\nu}\left(i_0,f\right) \asymp v^{n}$, where the implicit constants depend on $n$ and $p$. Hence, for any $N>0$, $$\mathcal{H}^s\!\left(W_{\tau, n}\left(p, i_0, f\right))\right) \, \ll \, \sum_{\nu > N} \sum_{\nu(\bm{r},q)\in \mathcal{A}_{\nu}\left(i_0,f\right)} \nu^{-\tau s} \, \ll \, \sum_{\nu > N} \nu^{-\tau s +n},$$ which is finite as soon as $s> (n+1)/\tau$, so that $\dim\left(W_{\tau, n}\left(p, i_0, f\right)\right)\le (n+1)/\tau$.

The proof that $\dim\left(W_{\tau, n}\left(p, i_0, f\right)\right)\ge (n+1)/\tau$ is very similar to the corresponding result for the limsup set $W_{\tau, n}(p)$ as defined in~(\ref{limsuppadique}) and shall therefore not be given~: this is due to the fact that $W_{\tau, n}\left(p, i_0, f\right)$ is itself a limsup set. For further details, the reader is referred to~\cite{aberhausdpadic} and~\cite{meastheorlawlimsup}.

This completes the proof of Theorem~\ref{thmprincipalliminfpadique}.

\renewcommand{\abstractname}{Acknowledgements}
\begin{abstract}
The author would like to thank his PhD supervisor Detta Dickinson for discussions which helped to develop ideas put forward. He would also like to thank his friend Jean--Matthieu George--Hanna for the good times spent together when thinking of all these problems. He is supported by the Science Foundation Ireland grant RFP11/MTH3084.
\end{abstract}

\bibliographystyle{plain}
\bibliography{liminf_sets}

\end{document}